\documentclass[12pt]{amsart} 
\usepackage{amssymb,latexsym, amsmath, amscd, array, graphicx} 

\long\def\forget#1\forgotten{} %

\swapnumbers 
\numberwithin{equation}{section} 

\def\m{\medskip}

\theoremstyle{plain} 
 
\newtheorem{thm}{Theorem}[section] 
\newtheorem{theorem}[thm]{Theorem}

\newtheorem{cor}[thm]{Corollary}

\newcommand\theoref{Theorem~\ref}

\newcommand\corref{Corollary~\ref} 
 
\newcommand\remref{Remark~\ref}

\def\secref{Section~\ref} 

\theoremstyle{definition} 
 
\newtheorem{definition}[thm]{Definition}

\newtheorem{rem}[thm]{Remark} 
\newtheorem{remark}[thm]{Remark}

\DeclareMathOperator{\inv}{{\rm inv}} 
\DeclareMathOperator{\DR}{{\rm DR}}

\def\1{\hbox{\rm\rlap {1}\hskip.03in{\rom I}}} 
\def\Bbbone{{\rm1\mathchoice{\kern-0.25em}{\kern-0.25em} 
{\kern-0.2em}{\kern-0.2em}I}} 

\def\m{\medskip}

\long\def\forget#1\forgotten{} %

\newcommand\ver[1]{\marginpar{\tiny Changed in Ver \VER}} 

\date{\today} 

\begin{document}  

\title[invariant contact structures]
{Invariant contact structures on $7$-dimensional nilmanifolds} 

\author[S. Kutsak]{Sergii Kutsak}

\address {Sergii Kutsak, Department of Mathematics, University of Florida,
358 Little Hall, Gainesville, FL 32601, USA}
\email{sergiikutsak@ufl.edu} 

\begin{abstract}
In this paper we give the list of all $7$-dimensional nilpotent real Lie algebras that admit a contact structure. Based on this list, we describe all $7$-dimensional nilmanifolds that admit an invariant contact structure.
\end{abstract}

\maketitle
   
\section{Introduction}

The goal of the paper is to describe all 7-dimensional nilmanifolds admitting an invariant contact structure, see \theoref{t:main}. To achieve this goal, we classify all 7-dimensional real nilpotent Lie algebra that admit a contact structure, see \corref{c:main}.  First, we recall some definitions.

\begin{definition}
 For a Lie algebra $\mathfrak{g}$ the upper central series is the increasing sequence of ideals  defined as follows:
$$ C_0(\mathfrak{g})=0, C_{i+1}(\mathfrak{g})=\{x \in \mathfrak{g}| [x,\mathfrak{g}] \subset C_i(\mathfrak{g})\}. $$
In other words $ C_{i+1}(\mathfrak{g}) $ is the inverse image under the canonical mapping of $ \mathfrak{g} $ onto $ \mathfrak{g}/C_i(\mathfrak{g}) $ of the center of $ \mathfrak{g}/C_i(\mathfrak{g}) $.

 The Lie algebra $ \mathfrak{g} $ is called { \em nilpotent} if there is an integer $k$ such that $ C_k(\mathfrak{g})= \mathfrak{g}.$ The minimal such $k$ is called the { \em index of nilpotency}, ~\cite{B}.
\end{definition}

A connected Lie group is nilpotent if and only if its Lie algebra is nilpotent.

\begin{definition}
A { \em nilmanifold} $M$ is a compact homogeneous space of the form $M=N/\Gamma$, where $N$ is a simply connected nilpotent Lie group and $\Gamma$ is a discrete cocompact subgroup in $N$, ~\cite{TO}. 
\end{definition}

For example, an $n$-dimensional torus $T^n=\mathbb{R}^n/\mathbb{Z}^n$ is obviously a nilmanifold.
If we consider the group $U_n(\mathbb{R})$ of upper triangular matrices having 1s along the diagonal then the quotient $ M=U_n(\mathbb{R})/U_n(\mathbb{Z}) $ is a nilmanifold, called the Heisenberg nilmanifold, where 
$U_n(\mathbb{Z}) \subset U_n(\mathbb{R}) $ is the set of matrices having integral entries.
\\ Let $ N/\Gamma $ be a nilmanifold and let $\mathfrak{n}$ be the Lie algebra of the Lie group $N$. It is well-known that the exponential map $ \rm exp:\mathfrak{n} \to N $ is a global diffeomorphism and the quotient map $ N \to N/\Gamma $ is the universal covering map. 
Hence every nilmanifold is the Eilenberg-MacLane space $ K(\Gamma,1) $. 
By Malcev's theorem a discrete group $\Gamma$ can be realized as the fundamental group of a nilmanifold if and only if it is a finitely presented  nilpotent torsion free group ~\cite{Ma}. 
 
Recall that a space $X$ is called {\em nilpotent} if the fundamental group $\pi_1(X)$ is nilpotent and the action of $\pi_1(X) $ on all homotopy groups $\pi_n(X), n\geq 1 $ is nilpotent.  It is clear that every nilmanifold $M=N/\Gamma$ is a nilpotent space since $\pi_1(M)$ is nilpotent and all higher homotopy groups $\pi_n(M), n \geq 2$ are trivial. Hence, it follows from the fundamental theorem of the rational homotopy theory ~\cite{DGMS} that two nilmanifolds have the same rational homotopy type if and only if the corresponding Chevalley-Eilenberg complexes are isomorphic.

\begin{theorem}[\cite{Ma}]\label{t:ma}
A simply connected nilpotent Lie group $N$ admits a discrete cocompact subgroup $\Gamma$ if and only if  there exists a basis $\{e_1, e_2, \dots, e_n\}$ of the Lie algebra $\mathfrak{n} $ of $N$ such that the structural constants $c_{ij}^k$ arising in brackets 
$$ \displaystyle{[e_i, e_j]=\sum_k{c_{ij}^ke_k}} $$
are rational numbers for all $i,j,k$. 
\end{theorem}

Let $ \mathfrak{g} $ be a Lie algebra with a basis $ \{X_1, \dots ,X_n\} $. Denote by $ \{x_1,\dots ,x_n\} $ the basis for $ \mathfrak{g}^* $ dual to $ \{X_1, \cdots ,X_n\} $. We obtain a differential $d$ on the exterior algebra $\Lambda\mathfrak{g}^*$
by defining it on degree 1 elements as 
$$ dx_k(X_i,X_j)=-x_k([X_i,X_j]) $$
and extending to $\Lambda\mathfrak{g}^*$ as a graded bilinear derivation. Then
$$ 
\displaystyle{[X_i,X_j]=\sum_{l}{c_{ij}^lX_l}} 
$$
where $c_{ij}^l$ are the structure constants of $ \mathfrak{g} $, and it follows from duality that 
$$ 
dx_k(X_i,X_j)=-c_{ij}^k.
$$
Hence on generators the differential is expressed as
$$ 
dx_k=-\sum_{i<j}{c_{ij}^kx_ix_j}. 
$$
Note that the condition $d^2=0$ is equivalent to the Jacobi identity in the Lie algebra. We call the differential graded algebra
$$ (\Lambda\mathfrak{g}^*,d) $$
the {\em Chevalley-Eilenberg } complex of the Lie algebra $\mathfrak{g} $.

A {\em contact structure} on a manifold $M$ of odd dimension $2n+1$ is a completely non-integrable $2n$-dimensional tangent plane distribution $\xi$. In the coorientable case the distribution may be defined by a differential 1-form $\alpha$ as $\xi=\ker \alpha$. Then the non-integrability condition can be expressed by the inequality $ \alpha \wedge (d\alpha)^n \neq 0 $, i.e. the form $ \alpha \wedge (d\alpha)^n$ is nowhere zero. Such differential form $\alpha$ is called a { \em contact differential form}. A contact structure can be viewed as an equivalence class of contact differential forms, where two forms $\alpha$ and $\alpha'$ are equivalent if and only if $\alpha'=f\alpha$ where $f$ is a nowhere zero smooth function on $M$. A { \em contact manifold} is a pair $(M^{2n+1},\eta)$ where $M$ is a smooth manifold of dimension $2n+1$ and $\eta$ is a contact structure on $M$,~\cite{E, MS}.

\begin{definition}
We say that a Lie algebra $\mathfrak{g}$ of dimension $2n+1$ admits a contact structure if there is an element $\eta$  of degree $1$ in the Chevalley-Eilenberg complex $ (\Lambda\mathfrak{g}^*,d) $ of $\mathfrak{g}$ such that $ \eta \wedge (d\eta)^n \neq 0.$ 
\end{definition}

\begin{definition}
We say that a contact structure  (resp. a contact form)  on a nilmanifold $N/\Gamma$ is left (right)  {\em invariant} if the contact structure (resp. a contact form) is invariant with respect to the left (resp. right) $N$-action on $N/\Gamma$. 
\end{definition} 

\begin{rem}\label{r:invar}
Since every invariant diferential form on a nilmanifold $ N/\Gamma $ is completely determined by its values on the Lie algebra  $\mathfrak{n}$ of $N$, we conclude that a nilmanifold $ N/\Gamma $ admits an invariant contact structure if and only if $\mathfrak{n}$ admits a contact structure.
\end{rem}
Note that any invariant contact structure on $N/\Gamma$ is coorientable.  

\begin{remark} 
Clearly, the kernel of invariant contact form gives us an invariant contact structure. Conversely, for any contact form $\alpha$ on $N/\Gamma$ we can perform the averaging of the lift of $\alpha$ to $N$ rescaled down by a nowhere zero function with the finite integral and get an invariant contact form $\overline \alpha$. Furthermore, $\overline\alpha$ and $\alpha$ yield the same contact structure provided $\alpha$ defines an invariant contact structure. In particular, a nilmanifold $M$ admits an invariant contact structure if and only if $M$ admits an invariant contact form. 
\end{remark}

To find all 7-dimensional nilmanifolds that admit an invariant contact structure we used the classification of 7-dimensional indecomposable nilpotent Lie algebras. Many attempts have been done on this topic. (We shall see below in \secref{s:decomp} that the decomposable nilpotent Lie algebras do not admit  a contact structure.) There are a few lists available: Safiullina  (1964, over $\mathbb{C}$)~\cite{S}, Romdhani (1985, over $\mathbb{R}$ and $\mathbb{C}$)~\cite{R}, Seeley (1988, over $\mathbb{C}$)~\cite{Se1}, Ancochea and Goze (1989, over $\mathbb{C}$)~\cite{AG}. The lists above are obtained by using different invariants. Carles ~\cite{C} introduced a new invariant - the weight system, compared the lists of Safiullina, Romdhani and Seeley, and found mistakes and omissions in all of them. Later in 1993 Seeley incorporated all the previous results and published his list over $\mathbb{C}$,~\cite{Se2}. In 1998 Gong used the Skjelbred-Sund method to classify all 7-dimensional nilpotent Lie algebras over $\mathbb{R}$, ~\cite{G}. We will use Gong's classification with some corrections from the list of Magnin,~\cite{M}.

\m To find an invariant contact structure on a nilmanifold, we first apply~\remref{r:invar}. So, we must check whether the nilpotent Lie algebras from the Gong's list  admit a contact structure. To achieve this goal,  we use the Sullivan minimal model theory and~\corref{c:ce} (see sections~\ref{s:min} and \ref{s:main}). Next, for a given 7-dimensional nilpotent Lie algebra having a contact structure,  we must check whether the corresponding simply connected Lie group contains a discrete cocompact  subgroup. For this we use the Malcev's criterion~(\theoref{t:ma}). 

\section{Acknowledgements}
I would like to thank my advisors Alexander Dranishnikov and Yuli Rudyak for many useful discussions and comments on this work.

\section{Preliminaries: Minimal Models}\label{s:min}

References for the minimal model theory are~\cite{DGMS, FHT, TO}.

\begin{definition}
A { \em differential graded algebra}  (DGA) over a field $k$ of characteristic $0$ is a graded vector space 
$ \displaystyle{A=\oplus_{i=0}^{\infty}A^i} $ with multiplications $ A^p \otimes A^q \to A^{p+q} $ which are associative and commutative in graded sense
$$ a \cdot b=(-1)^{|a|\cdot|b|}b \cdot a $$
where $|a|$ and $|b|$ denote the degrees of $a$ and $b$ respectively in the underlying graded vector space. The DGA $A$ also possesses a differential $ d:A^p \to A^{p+1} $ which is a graded derivation
$$ d(a \cdot b)=(da)\cdot b+(-1)^{|a|}a\cdot (db) $$
satisfying $d^2=0.$ A { \em morphism }in the category of the DGA's is an algebra homomorphism which commutes with the differential and respects grading.
\end{definition}

\begin{definition}
A DGA $ (\mathcal{M},d) $ is called a { \em model } for a DGA $(A,d_A)$  if there exists a DGA-morphism
$$ f:(\mathcal{M},d) \to (A,d_A) $$
that induces an isomorphism on cohomology. If $ (\mathcal{M},d) $ is freely generated in the sense that $ \mathcal{M}=\Lambda V $ for a graded vector space  $ \displaystyle{V=\oplus V^i } $ then it is called a free model for $(A, d_A)$, where the notation $ \Lambda V $means that, as a graded algebra, $\mathcal{M}$ is the tensor product of a polynomial algebra on even elements $ V^{\rm even} $ and an exterior algebra on odd elements $V^{\rm odd}$.
\end{definition}

\begin{definition}
A DGA $ (\mathcal{M},d) $ is called a {\em minimal model} of $(A,d_A)$ if the following conditions are satisfied:
\\ (i) $( \mathcal{M},d)=(\Lambda V,d) $ is a free model for $(A,d_A)$;
\\ (ii) $d$ is indecomposable in the following sense: there exists a basis $\{x_\mu: \mu \in I \} $ for some well-ordered index set $ I $ such that $\rm deg (x_{\lambda})\leq \rm deg (x_{\mu}) $ if $ \lambda<\mu $, and each $ d(x_{\lambda}) $ is expressed in terms of $ x_{\mu} $,  $\mu<\lambda$. It follows that $dx_{\mu}$ does not have a linear part.
\end{definition}

\begin{definition}
A minimal model of a smooth manifold $M$ is a minimal model of its de Rham DGA.
\end{definition}

We need the following fact,~\cite{CE}. Let $M=N/\Gamma$ be a nilmanifold. Then the complex of differential forms on $M$ can be identified with the the complex 
of differential forms on $N$ which are right invariant by the elements of $\Gamma$.

\begin{theorem}[\cite{N}]
Let $M=N/\Gamma$ be a nilmanifold.
The natural inclusion of the complex of right invariant differential forms on $N$ into the complex of the differential forms on $N/\Gamma$
$$ \Omega_{\DR}^{\inv}(N) \to \Omega_{\DR}(N/\Gamma)$$
induces an isomorphism on the cohomology level.
\end{theorem}

\begin{cor}[\cite{TO}]\label{c:ce}
The minimal model of a compact nilmanifold $N/\Gamma$ is isomorphic to the Chevalley-Eilenberg complex of the Lie algebra 
$ \mathfrak{n} $ of $ N $.
\end{cor}

\section{Contact structures on decomposable nilpotent Lie algebras}\label{s:decomp}

\begin{definition} 

A Lie algebra $ \mathfrak{n} $ is called { \em decomposable} if it can be represented as a direct sum of its ideals.
\end{definition}

\begin{theorem}\label{t:indecom} 
Let $N/\Gamma$ be a nilmanifold of dimension $2n+1, n \geq 0$. If the Lie algebra $ \mathfrak{n} $ of the Lie group $N$ is decomposable then the nilmanifold $N/\Gamma$ does not admit an invariant contact structure.
\end{theorem}
\begin{proof}
 Let $ \mathfrak{n}=V \oplus W $ where $V$ and $W$ are ideals in $ \mathfrak{n} $ so that $ [X,Y]=0 $ for all $ X \in V, Y \in W$. Let ${X_1,\dots,X_s},{Y_1,\dots,Y_t}$ be the basis of $V$ and $W$, respectively. Denote by ${x_1,\dots,x_s},{y_1,\dots,y_t}$
the dual basis of $V^*$ and $W^*$, respectively. Consider the Chevalley-Eilenberg complex $ (\Lambda\mathfrak{n}^*, \delta) $ of the Lie algebra $ \mathfrak{n} $. Then
$$ \displaystyle{\delta x_p=\sum_{i<j}{c_{ij}^px_ix_j}, j<p, p=1,\dots,s} $$
$$ \displaystyle{\delta y_q=\sum_{l<m}{c_{lm}^qy_ly_m}, m<q, q=1,\dots,t}, $$
because $ [X_p,Y_q]=0$ for all $ p=1,\dots,s,q=1,\dots,t $.
Hence, without loss of generality we may assume that $ \{x_1,x_2,\dots,x_s,y_1,y_2,\dots,y_t\} $ is the basis of $ \mathfrak{n}^* $ that is ordered to satisfy the indecomposability condition of the minimal model $ (\Lambda\mathfrak{n}^*, \delta) $. Consider an arbitrary $1$-form 
$$\displaystyle{\omega=\sum_{i=1}^{s}{a_ix_i}+\sum_{j=1}^{t}{b_jy_j},\quad a_i,b_j \in \mathbb{R}}$$
on $ \mathfrak{n} $. Then the derivative $\delta\omega$ will contain no terms with $x_s$ or $y_t$. Thus, the $2n$-form $(\delta\omega)^n$ vanishes because it is a linear combination of terms each of which is a product of $2n$ $1$-forms from the 
set 
$$\{x_1,x_2,\dots,x_{s-1},y_1,y_2,\dots,y_{t-1}\} $$
of cardinality $2n-1$. It follows that the nilmanifold $N/\Gamma$ cannot admit an invariant differential form $\eta$ such that $ \eta \wedge (d\eta)^n \neq 0 $.
\end{proof}

\section{Contact structures on indecomposable nilpotent Lie algebras}\label{s:main}

 Now we discuss the existence of  contact structures on indecomposable Lie algebras. We have proved that there are no contact structures on decomposable nilpotent Lie algebras, see \theoref{t:indecom} above. All 7-dimensiomal indecomposable nilpotent Lie algebras over the real field have been classified, ~\cite{G,M}. 
Over reals there are 138 non-isomorphic 7-dimensiomal indecomposable nilpotent Lie algebras and, in addition, 9 one-parameter families.
For the one-parameter families, a parameter $\lambda$  is used to denote a structure constant
that may take on arbitrary values (with some exceptions) in $\mathbb{R}.$  An invariant $ I(\lambda) $ is given
for each family in which multiple values of $\lambda$ yield isomorphic Lie algebras, i.e., if $ I(\lambda_1) = I(\lambda_2)$,
then the two corresponding Lie algebras are isomorphic, and conversely. All one-parameter families admit a contact structure (with some restrictions on a parameter $\lambda$) and only $35$ Lie algebras admit a contact structure. We conduct a detailed computation for the algebra (1357C) and the one-parameter family (147E).  Here we use the same notation for the Lie algebras as in ~\cite{G}, where the Lie algebras are listed in accordance with their upper central series dimensions.  For instance, the Lie algebras that have the upper central series dimensions 1,4,7 are listed as follows: (147A), (147B), (147C), etc. For the sake of brevity we will drop the sign of the wedge product (for instance, $x_ix_j $ means $ x_i \wedge x_j $ ).

\begin{rem}
Let $\mathfrak{n}$ be a 7-dimensional nilpotent Lie algebra with the basis $\{X_i:i=1, \dots,7\} $ and the corresponding dual basis $\{x_i:i=1, \dots,7\} $. Consider the Chevalley-Eilenberg complex $ (\Lambda \mathfrak{n}^*,\delta)$ of the Lie algebra $\mathfrak{n} $. If one cannot find nonzero structural constants $ c_{i_1j_1}^{k_1},c_{i_2j_2}^{k_2},c_{i_3j_3}^{k_3}$ such that $\{i_1,j_1,i_2,j_2,i_3,j_3\}=\{1,2,3,4,5,6\} $ then $\mathfrak{n}$
does not admit a contact structure because for every $1$-form 
$$ \displaystyle{\omega=\sum{a_lx_l}, a_l \in \mathbb{R}, l=1, \dots, 7} $$
 on $\mathfrak{n}$ the $6$-form $(\delta \omega)^3$ contains at most five distinct elements from the basis $ \{x_1, \dots ,x_7\} $ and therefore $(\delta \omega)^3=0$. For this reason $103$ $7$-dimensional indecomposable nilpotent Lie algebras that are listed in  ~\cite{G}, pp.$53-64$ do not admit a contact structure and therefore we do not list them in our paper. Thus, if a 7-dimensional simply connected nilpotent  Lie group $N$ admits a discrete cocompact subgroup $ \Gamma$ and the Lie algebra $\mathfrak{n}$ of $ N $ falls into this list of $103$ $7$-dimensional indecomposable nilpotent Lie algebras then the nilmanifold $ N/\Gamma $ does not admit an invariant contact structure.
\end{rem}

Now we show that the Lie algebra (1357C) and the one-parameter family (147E) admit a contact structure.
First we consider the Lie algebra (1357C). Let  $\{X_i:i=1, \dots,7\} $ be a basis of the Lie algebra (1357C) and  $\{x_i:i=1, \dots,7\} $ be the corresponding dual basis. Note that $(1357)$ stands for the upper central series dimensions. The nontrivial Lie brackets of the Lie algebra (1357C) are defined as follows: \\

\begin{tabular}{ l l l l}
 $[X_1,X_2]=X_4,$ & $[X_1,X_4]=X_5,$ & $ [X_1,X_5]=X_7,$ & $[X_3,X_6]=X_7,$ \\
 $ [X_2,X_3]=X_5, $ & $ [X_2,X_4]=X_7,$ & $ [X_3,X_4]=-X_7. $ \\
  
\end{tabular} \vspace{0.4cm}

Then we can find the differential of the Chevalley-Eilenberg model of the Lie algebra (1357C): \vspace{0.4cm} \\ 
\begin{tabular}{ l l }
  $ dx_i=0, i=1,2,3, $ & $ dx_4=x_1x_2 $ \\ 
  $ dx_5=x_2x_3+x_1x_4,$ & $ dx_7=x_1x_5+x_2x_4-x_3x_4+x_3x_6$ . 
  
\end{tabular}  
\vspace{0.4cm}
\\ Let $ \eta=x_7$. A straightforward computation shows that $\eta \wedge (d\eta)^3=6x_1x_2x_3x_4x_5x_6x_7 \neq 0. $ Hence the algebra (1357C) admits a contact structure. Similarly one can see that there are $34$ other algebras that admit a contact structure and the contact form $\eta$ is given by $\eta=x_7$, except for the algebra (12457L) for which $\eta=x_7+x_6$.

\m Consider the one-parameter family (147E) with invariant $ I(\lambda)=\frac{(1-\lambda+\lambda^2)^3}{\lambda^2(\lambda-1)^2}$, $\lambda \neq 0,1$ and the Lie bracket defined as follows: \\

\begin{tabular}{ l l l  }
 $ [X_1,X_2]=X_4, $ & $ [X_1,X_3]=-X_6, $ & $ [X_1,X_5]=-X_7,$ \\
 $[X_2,X_3]=X_5,$ & $[X_2,X_6]=\lambda X_7,$ & $ [X_3,X_4]=(1-\lambda)X_7.$ \\
\end{tabular}  \vspace{0.4cm}

Then we can find the differential of the Chevalley-Eilenberg model of the one-parameter family (147E): \vspace{0.4cm} \\ 
\begin{tabular}{ l l l }
  $ dx_i=0, i=1,2,3, $ & $ dx_4=x_1x_2 $ & $dx_7=\lambda x_2x_6+(1-\lambda)x_3x_4-x_1x_5,$  \\ 
  $ dx_6=-x_1x_3,$ & $ dx_5=x_2x_3.$ \\
  
\end{tabular}  
\vspace{0.4cm}

 Let $ \eta=x_7$. Then $\eta \wedge (d\eta)^3=6\lambda(1-\lambda)x_1x_2x_3x_4x_5x_6x_7 \neq 0,$ for $ \lambda \neq 0,1. $ Hence the one-parameter family (147E) admits a contact structure for any $\lambda \neq 0,1.$ Similarly one can see that the remaining $8$ families admit a contact structure and the contact form $\eta$ is given by $\eta=x_7$.  Eventually, we get the following result.

\begin{theorem}\label{t:main} The following is a complete and non-redundant list $\mathcal{L}$ of all $7$-dimensional nilpotent Lie algebras that admit a contact structure:

\rm
\begin{center}
Upper Central Series Dimensions $ (17) $
\end{center}
\begin{tabular}{ l l l l }
(17) & $ [X_1,X_2]=X_7,$  & $ [X_3,X_4]=X_7, $  & $ [X_5,X_6]=X_7.$  \\
\end{tabular}

\begin{center}
Upper Central Series Dimensions $ (157) $
\end{center}
\begin{tabular}{ l l l l }
(157) & $ [X_1,X_2]=X_3, $ & $ [X_1,X_3]=X_7, $ & $ [X_2,X_4]=X_7,$ \\
 & $ [X_5,X_6]=X_7.$ \\
\end{tabular} 

\begin{center}
Upper Central Series Dimensions $ (147) $
\end{center}
\begin{tabular}{ l l l l }
(147A) & $ [X_1,X_2]=X_4,$ & $[X_1,X_3]=X_5,$ & $ [X_1,X_6]=X_7,$ \\
    & $[X_2,X_5]=X_7,$ & $    [X_3,X_4]=X_7. $ \\
(147A$_1$) & $[X_1,X_2]=X_4,$ & $[X_1,X_3]=X_5,$ & $ [X_1,X_6]=X_7,$ \\
  & $ [X_2,X_4]=X_7,$ & $  [X_3,X_5]=X_7. $ \\
(147B) & $[X_1,X_2]=X_4,$ & $[X_1,X_3]=X_5,$ & $ [X_1,X_4]=X_7,$ \\
  & $ [X_2,X_6]=X_7,$ & $   [X_3,X_5]=X_7. $ \\
(147D) & $[X_1,X_2]=X_4,$ & $[X_1,X_3]=-X_6,$ & $ [X_1,X_5]=X_7,$ \\
  & $ [X_1,X_6]=X_7,$ & $  [X_2,X_3]=X_5,$ & $ [X_2,X_6]=X_7, $ \\
    & $ [X_3,X_4]=-2X_7. $ \\
(147E) &  $ [X_1,X_2]=X_4, $ & $ [X_1,X_3]=-X_6, $ & $ [X_1,X_5]=-X_7,$ \\
   & $[X_2,X_3]=X_5,$ & $[X_2,X_6]=\lambda X_7,$ & $ [X_3,X_4]=(1-\lambda)X_7,$ \\
     & $ I(\lambda)=\frac{(1-\lambda+\lambda^2)^3}{\lambda^2(\lambda-1)^2}$, & $\lambda \neq 0,1.$ \\    
\end{tabular} 

\begin{center}
\end{center}
\begin{tabular}{ l l l l }

(147E$_1$) & $ [X_1,X_2]=X_4,$ & $[X_2,X_3]=X_5,$ & $ [X_1,X_3]=-X_6,$ \\
  & $ [X_1,X_6]=-\lambda X_7, $ & $ [X_2,X_5]=\lambda X_7,$ & $[X_2,X_6]=2 X_7, $ \\
    & $  [X_3,X_4]=-2X_7, $ & $ \lambda > 1. $ \\
\end{tabular}

\begin{center}
Upper Central Series Dimensions $ (1457) $
\end{center}
\begin{tabular}{ l l l l }
(1457B) & $[X_1,X_i]=X_{i+1},i=2,3,$ & $[X_1,X_4]=X_7,$ & $ [X_2,X_3]=X_7,$ \\
  & $  [X_5,X_6]=X_7.$ \\
\end{tabular} 

\begin{center}
Upper Central Series Dimensions $ (137) $
\end{center}
\begin{tabular}{ l l l l }
(137B) & $[X_1,X_2]=X_5,$ & $[X_1,X_5]=X_7,$ & $ [X_2,X_4]=X_7,$ \\
  & $ [X_3,X_4]=X_6,$ & $  [X_3,X_6]=X_7.  $ \\
(137B$_1$) & $ [X_1,X_3]=X_5,$ & $[X_1,X_4]=X_6,$ & $ [X_1,X_5]=X_7,$ \\
    &  $ [X_2,X_3]=-X_6,$ & $ [X_2,X_4]=X_5,$ & $[X_2,X_6]=X_7,$ \\ 
     & $ [X_3,X_4]=X_7. $ \\
(137D) & $[X_1,X_2]=X_5,$ & $[X_1,X_4]=X_6,$ & $ [X_1,X_6]=X_7,$ \\
  & $ [X_2,X_3]=X_6,$ & $  [X_2,X_4]=X_7, $ & $[X_3,X_5]=-X_7. $ \\ 
\end{tabular} 

\begin{center}
Upper Central Series Dimensions $ (1357) $
\end{center}

\begin{tabular}{ l l l l }
(1357A) & $[X_1,X_2]=X_4,$ & $[X_1,X_4]=X_5,$ & $ [X_1,X_5]=X_7,$ \\
  & $ [X_2,X_3]=X_5,$ & $  [X_2,X_6]=X_7,$ & $ [X_3,X_4]=-X_7. $ \\
(1357C) & $[X_1,X_2]=X_4,$ & $[X_1,X_4]=X_5,$ & $ [X_1,X_5]=X_7,$ \\
  & $ [X_2,X_3]=X_5, $ & $ [X_2,X_4]=X_7,$ & $ [X_3,X_4]=-X_7, $ \\
    & $[X_3,X_6]=X_7.$ \\
(1357D) & $[X_1,X_2]=X_3,$ & $[X_2,X_i]=X_{i+2},i=3,4,$ & $ [X_1,X_6]=X_7,$ \\
  & $ [X_2,X_5]=X_7,$ & $  [X_3,X_4]=X_7. $ \\ 
(1357F) & $[X_1,X_2]=X_3,$ & $[X_2,X_i]=X_{i+2},i=3,4,$ & $ [X_1,X_3]=X_7,$ \\
  & $ [X_2,X_5]=X_7,$ & $  [X_4,X_6]=-X_7. $ \\  
   (1357F$_1$)  & $ [X_1,X_2]=X_3,$ & $[X_2,X_i]=X_{i+2}, i=3,4,$ & $ [X_1,X_3]=X_7,$ \\
  & $ [X_2,X_5]=X_7,$ & $  [X_4,X_6]=X_7. $   \\ 
(1357H) & $[X_1,X_2]=X_3,$ & $[X_1,X_4]=X_6, $ & $ [X_1,X_6]=X_7,$ \\
  & $ [X_2,X_3]=X_5, $ & $ [X_2,X_5]=X_7,$ & $ [X_2,X_6]=X_7, $ \\
    & $ [X_3,X_4]=-X_7.$ \\
(1357J) & $[X_1,X_2]=X_3,$ & $[X_1,X_3]=X_7,$ & $ [X_1,X_4]=X_6,$ \\ 
  & $ [X_2,X_3]=X_5, $ & $ [X_2,X_5]=X_7,$ & $ [X_4,X_6]=X_7. $ \\
 (1357L) & $[X_1,X_2]=X_3,$ & $ [X_1,X_i]=X_{i+2},i=3,4,5,$ & $ [X_2,X_3]=X_7, $ \\
  & $  [X_2,X_4]=X_5,$ & $ [X_2,X_6]=\frac{1}{2}X_7,$ & $  [X_3,X_4]=\frac{1}{2}X_7.$ \\
  (1357M) &  $[X_1,X_2]=X_3,$ & $ [X_1,X_i]=X_{i+2},i=3,4,5,$ & $ [X_2,X_4]=X_5, $ \\
  & $  [X_2,X_6]=\lambda X_7,$ & $ [X_3,X_4]=(1-\lambda)X_7,$ & $ \lambda \neq 0,1. $ \\   
 (1357N) & $[X_1,X_2]=X_3,$ & $ [X_1,X_i]=X_{i+2},i=3,4,5,$ & $ [X_2,X_3]=\lambda X_7,$ \\
  & $ [X_2,X_4]=X_5,$ & $ [X_3,X_4]=X_7,$ & $ [X_4,X_6]=X_7.$ \\   
   (1357P) & $[X_1,X_2]=X_3,$ & $ [X_1,X_i]=X_{i+2},i=3,5,$ & $ [X_2,X_3]=X_6,$ \\
  & $  [X_2,X_4]=X_5,$ & $ [X_2,X_6]=X_7,$ & $  [X_3,X_4]=X_7.$ \\  
  \end{tabular} 
\begin{center}
\end{center}
\begin{tabular}{ l l l l }

   \end{tabular}

\begin{center}
\end{center}
\begin{tabular}{ l l l l } 
  (1357P$_1$) & $ [X_1,X_2]=X_3,$ & $ [X_1,X_i]=X_{i+2}, i=3,5,$ & $ [X_2,X_3]=X_6,$ \\
  & $  [X_2,X_4]=X_5,$ & $[X_2,X_6]=-X_7,$ & $  [X_3,X_4]=X_7.  $   \\   
   (1357QRS$_1$)  & $ [X_1,X_2]=X_3,$ & $ [X_1,X_3]=X_5,$ & $ [X_1,X_4]=X_6,$ \\
  & $  [X_1,X_5]=X_7,$ & $[X_2,X_3]=-X_6,$ & $  [X_2,X_4]=X_5,$ \\
    & $ [X_2,X_6]=\lambda X_7,$ & $ [X_3,X_4]=(1-\lambda)X_7,  $ & $ I(\lambda)=\lambda+\lambda^{-1},$ \\ 
    & $ \lambda \neq 0,\pm 1. $   \\  
 (1357R) & $[X_1,X_2]=X_3,$ & $[X_1,X_3]=X_5,$ & $ [X_1,X_6]=X_7,$ \\
  & $ [X_2,X_3]=X_6,$ & $  [X_2,X_4]=X_6,$ & $ [X_2,X_5]=X_7, $ \\
    & $ [X_3,X_4]=X_7.$ \\     
(1357S) &  $[X_1,X_2]=X_3,$ & $[X_1,X_3]=X_5,$ & $ [X_1,X_5]=X_7,$ \\
  & $ [X_1,X_6]=X_7,$ & $  [X_2,X_3]=X_6,$ & $ [X_2,X_4]=X_6,$ \\
    & $ [X_2,X_5]=X_7,$ & $[X_2,X_6]=\lambda X_7,$ & $ [X_3,X_4]=X_7,$ \\ 
      & $ \lambda \neq 0,1. $ \\
\end{tabular}

\begin{center}
Upper Central Series Dimensions $ (13457) $
\end{center}
\begin{tabular}{ l l l l }
(13457C) & $[X_1,X_i]=X_{i+1}, i=2,3,4,$ & $[X_1,X_6]=X_7,$ & $ [X_2,X_5]=X_7, $ \\ 
    & $[X_3,X_4]=-X_7.  $ \\
(13457E) & $[X_1,X_i]=X_{i+1}, i=2,3,4, $ & $[X_1,X_6]=X_7,$ & $ [X_2,X_3]=X_5, $ \\
  & $ [X_2,X_5]=X_7, $ & $[X_3,X_4]=-X_7. $ \\
(13457G) & $[X_1,X_i]=X_{i+1}, i=2,3,4,$ & $[X_1,X_6]=X_7,$ & $ [X_2,X_3]=X_6, $ \\
  & $ [X_2,X_4]=X_7,$ & $ [X_2,X_5]=X_7,$ & $ [X_3,X_4]=-X_7. $ \\
(13457I) & $[X_1,X_i]=X_{i+1}, i=2,3,4,$ & $[X_1,X_5]=X_7, $ & $[X_2,X_3]=X_6,$ \\
  & $ [X_2,X_5]=X_7,$ & $ [X_2,X_6]=X_7,$ & $ [X_3,X_4]=-X_7. $ \\
\end{tabular} 

\begin{center}
Upper Central Series Dimensions $ (12457) $
\end{center}
\begin{tabular}{ l l l l }
(12457D) & $[X_1,X_i]=X_{i+1}, i=2,3, $& $[X_2,X_6]=X_7,$ & $ [X_2,X_5]=X_6,$ \\
  & $[X_1,X_i]=X_{i+2}, i=4,5, $ & $ [X_3,X_4]=-X_7.  $ \\
(12457E) & $[X_1,X_i]=X_{i+1}, i=2,3, $ & $ [X_1,X_4]=X_6,$ & $ [X_1,X_6]=X_7, $ \\
    & $[X_2,X_3]=X_6,$ & $ [X_2,X_4]=X_7,$ & $  [X_2,X_5]=X_6,$ \\
      & $ [X_3,X_5]=X_7.  $ \\
(12457G) & $[X_1,X_i]=X_{i+1}, i=2,3,$ & $ [X_1,X_4]=X_6,$ & $ [X_1,X_5]=X_7, $ \\
  & $[X_2,X_i]=X_{i+1},i=5,6, $ & $ [X_2,X_3]=X_6,$ & $  [X_3,X_4]=-X_7.  $ \\
(12457I) & $[X_1,X_i]=X_{i+1}, i=2,3,5,6, $ & $ [X_3,X_4]=X_7, $ & $ [X_2,X_5]=X_7, $ \\
  & $[X_2,X_j]=X_{j+2}, j=3,4.$ \\
(12457J) & $[X_1,X_i]=X_{i+1}, i=2,3,5,6, $ & $[X_1,X_4]=X_7,$ & $ [X_2,X_3]=X_5, $ \\
  & $[X_2,X_4]=X_6,$ & $ [X_2,X_5]=X_7,$ & $  [X_3,X_4]=X_7.  $ \\
(12457J$_1$) & $ [X_1,X_i]=X_{i+1}, i=2,3,5,6, $ & $ [X_1,X_4]=X_7,$ & $ [X_2,X_5]=-X_7, $ \\
  & $  [X_2,X_j]=X_{j+2}, j=3,4,   $ & $ [X_3,X_4]=X_7 $ \\
 (12457L) & $[X_1,X_i]=X_{i+1}, i=2,3,5,6, $ & $[X_3,X_4]=X_7,$ & $ [X_2,X_6]=X_7,$ \\
  & $[X_2,X_j]=X_{j+2}, j=3,4,$ & $[X_3,X_5]=-X_7. $ \\  
  \end{tabular}
  
  \begin{center}
  \end{center}
  \begin{tabular}{ l l l l }
 (12457N) & $[X_1,X_i]=X_{i+1}, i=2,3,5,6,$ & $  [X_1,X_4]=X_7,$ & $ [X_2,X_3]=X_5,$ \\
  & $[X_2,X_4]=X_6, $ & $[X_2,X_5]=\lambda X_7,$ & $ [X_2,X_6]=X_7,$ \\
    & $  [X_3,X_4]=X_7,$ & $ [X_3,X_5]=-X_7,  $ & $ I(\lambda)=\lambda + \lambda^{-1}, $ \\
      & $ \lambda \neq 0. $ \\ 
(12457N$_2$) & $ [X_1,X_i]=X_{i+1},i=2,3,$ & $ [X_1,X_4]=-X_6,$ & $ [X_1,X_5]=X_7,  $ \\
  & $[X_1,X_6]=X_7,$ & $[X_2,X_3]=X_5,$ & $[X_2,X_4]=X_7,$  \\
    & $[X_2,X_5]=-X_6+\lambda X_7,$ & $[X_3,X_5]=-X_7,$ & $\lambda \ge 0, \lambda \neq 1.$  \\
\end{tabular} 

\begin{center}
Upper Central Series Dimensions $ (12357) $
\end{center}
\begin{tabular}{ l l l l }
(12357C) & $[X_1,X_2]=X_4,$ & $ [X_1,X_i]=X_{i+1},i=4,5,6,$ & $ [X_2,X_3]=X_5,$ \\
  & $  [X_2,X_4]=X_7, $ & $ [X_3,X_4]=-X_6, $ & $[X_3,X_5]=-X_7 . $ \\
\end{tabular}

\begin{center}
Upper Central Series Dimensions $ (123457) $
\end{center}
\begin{tabular}{ l l l l }
(123457C) & $[X_1,X_i]=X_{i+1}, 2 \leq i \leq 6,$ & $ [X_2,X_5]=X_7,$ & $ [X_3,X_4]=-X_7 , $ \\
(123457F) & $[X_1,X_i]=X_{i+1}, 2 \leq i \leq 5, $ & $[X_1,X_6]=X_7,$ & $ [X_2,X_3]=X_6, $ \\
  & $ [X_2,X_4]=X_7, $ & $[X_2,X_5]=X_7,$ & $ [X_3,X_4]=-X_7.  $ \\
(123457I) & $ [X_1,X_i]=X_{i+1}, 2 \leq i \leq 5,$ & $ [X_1,X_6]=X_7,$ & $ [X_2,X_3]=X_5, $ \\
  & $ [X_3,X_4]=(1-\lambda)X_7,$ & $ [X_2,X_5]=\lambda X_7,$ & $ [X_2,X_4]=X_6,   $  \\
    & $ \lambda \neq 0,1. $ \\
\end{tabular} 
\end{theorem}

\begin{cor}\label{c:main}
For a $7$-dimensional nilpotent Lie algebra $\mathfrak{n}$ we denote by $ \mathcal{C}(\mathfrak{n}) $ the set of all nilmanifolds $N/\Gamma$ such that $\mathfrak{n}$ is isomorphic to the Lie algebra of a Lie group $N$.  Then $\bigcup_{\mathfrak{n} \in \mathcal{L}}  \mathcal{C}(\mathfrak{n}) $ is the set of all $7$-dimensional nilmanifolds that admit an invariant contact structure.
\end{cor}

\forget
\begin{rem} Note that two nilmanifolds correspond to isomorphic rational Lie algebras if and only if they are finite-sheeted coverings for a third nilmanifold, ~\cite{Ma}.
\end{rem}
\forgotten

\m To see which of these Lie algebras yield a nilmanifold, we must check whether a Lie algebra from the list admits a basis with respect to which all structure constants are rational. From the above list, we see that, definitely, the following Lie algebras satisfy this condition: $(17), (157), $ $(147(A, A_1 ,B,D)), (1457), (137),   (1357(A, C, D, F,  F_1, H, J, L,P, P_1,  R))$, \\ $(13457), $ $(12457(D,E, G, I, J, J_1, L)),  (12357), $ $(123457(C,F))$. 

Concerning the remaining cases, note that the one-parameter families of Lie algebras with a rational parameter $\lambda$ give us the nilmanifolds. The case of an irrational parameter $\lambda$ is not clear to us. In principle, it can happen that such  a Lie algebra has a basis with respect to which all  structure constants are rational. This needs more research.

\section{The case of $\dim <7$}
 For completeness we discuss the existence of invariant contact structures on nilmanifolds of dimension $3$ and $5$. Note that all nilpotent Lie algebras of dimension up to 6 have been classified  ~\cite{Gr}. In dimension $5$ there are $9$ nonisomorphic nilpotent Lie algebras and there are $3$ of them that admit a contact structure $\eta=x_5$. We shall use the same notation as in ~\cite{Gr}. \\
\begin{tabular}{ l l l l }
L$_{5,1}$ & $[X_1,X_2]=X_5,$ & $[X_3,X_4]=X_5,$ &   \\
L$_{5,3}$ & $[X_1,X_2]=X_4,$ & $[X_1,X_4]=X_5,$ & $[X_2,X_3]=X_5$  \\
L$_{5,6}$ & $[X_1,X_2]=X_3,$ & $[X_1,X_3]=X_4,$ & $[X_1,X_4]=X_5$  \\
  &  $[X_2,X_3]=X_5 $ \\
\end{tabular} 

In dimension $3$ there are only two non-isomorphic nilpotent Lie algebras. One of them is abelian and therefore does not admit a contact structure. Hence  a torus $ \displaystyle{T=\mathbb{R}^3/\mathbb{Z}^3}$ does not admit an invariant contact structure. Another $3$-dimensional nilpotent Lie algebra has nontrivial Lie bracket $[X_1,X_2]=X_3$ and therefore admits a contact structure $\eta=x_3$. Hence the Heisenberg manifold is the only $3$-dimensional nilmanifold that admits an invariant contact structure.

\end{document}